\documentclass[10pt,twoside]{article}
\usepackage{amsmath, amsthm, amsfonts, amssymb}
\usepackage{mathrsfs}
\usepackage[misc]{ifsym}
\usepackage[colorlinks,linkcolor=blue,anchorcolor=blue,citecolor=blue,urlcolor=black]{hyperref}
\usepackage{graphicx}
\usepackage{subfigure}
\usepackage{tikz}
\usepackage{enumerate}
\usepackage{anysize}
\usepackage{setspace}
\usepackage{float}
\usepackage{color}
\usepackage{chngcntr}
\usepackage[numbers,sort&compress]{natbib}
\everymath{\displaystyle}
\allowdisplaybreaks

\newfam\msbfam
\font\tenmsb=msbm5    \textfont\msbfam=\tenmsb \font\sevenmsb=msbm5
\scriptfont\msbfam=\sevenmsb \font\fivemsb=msbm5
\scriptscriptfont\msbfam=\fivemsb

\newfam\bigfam
\font\tenbig=msbm5 scaled \magstep2   \textfont\bigfam=\tenbig
\font\sevenbig=msbm7 scaled \magstep2 \scriptfont\bigfam=\sevenbig
\font\fivebig=msbm5 scaled \magstep2
\scriptscriptfont\bigfam=\fivebig

\numberwithin{equation}{section}
\newtheorem{theorem}{Theorem}[section]
\newtheorem{lemma}{Lemma}[section]

\newtheorem{remark}{Remark}[section]
\newtheorem{corollary}{Corollary}[section]
\newtheorem{definition}{Definition}[section]
\begin{document}
	\title{\bf Multilinear Fractional Integral Operators with Generalized Kernels }
	\author{\bf  Yan Lin,  Yuhang Zhao and Shuhui  Yang$^*$}
	\renewcommand{\thefootnote}{}
	\date{}
	\maketitle
	\footnotetext{This work was partially supported by the National 
		Natural Science Foundation of China No. 12071052 and Basic Research Foundation of China University of Mining and Technology (Beijing) -- Cultivation of Top Innovative Talents for Doctoral No. BBJ2023058.}
	\footnotetext{$^*$Corresponding author, E-mail:yangshuhui@student.cumtb.edu.cn}
	\begin{minipage}{13.5cm}
		{\bf Abstract}
		\quad
		In this article, we introduce a class of multilinear fractional integral operators with generalized kernels that are weaker than the Dini kernel condition. 
		We establish the boundedness of multilinear fractional integral operators with generalized kernels on weighted Lebesgue spaces and variable exponent Lebesgue spaces, as well as the boundedness of multilinear commutators generated by multilinear fractional integral operators with generalized kernels and {\rm BMO} functions. Even when the generalized kernels condition goes back to the Dini kernel condition, the conclusions on the commutators remain new.
	\end{minipage}
	
\section{Introduction}\label{sec1}
 \qquad  Serving as a generalization of the Hilbert transform to higher-dimensional Euclidean spaces, the Riesz transform, as a classical singular integral operator, is defined by setting, for any $f\in\mathscr{S}({\mathbb{R}}^n)$(the space of all Schwartz functions on ${\mathbb{R}}^n$) and $x\in \mathbb{R}^n$,
\begin{align*}
R_j(f)(x):=c_n\,{\rm p.v.}\int_{{\mathbb{R}}^n}\frac{x_j-y_j}{|x-y|^{n+1}}f(y)dy,\  1\leq j\leq n,
\end {align*}
where $c_n:=\Gamma(\frac{n+1}{2})/{\pi}^{\frac{n+1}{2}}$. The Riesz transforms have wide applications in fields such as image processing, signal processing, and partial differential equations, and are commonly used for tasks such as edge detection and texture analysis \cite{FGL2023,LA2010,ZL2012,ZZM2010}.

The Riesz potential, which plays a crucial role in harmonic analysis,  is also known as the fractional integral operator. Let $0<\alpha<n$ and the Riesz potential is defined by setting, for any $f\in\mathscr{S}({\mathbb{R}}^n)$ and $x\in {\mathbb{R}}^n$,
\begin{align*}
I_{\alpha}(f)(x):=\frac{1}{\gamma(\alpha)}\int_{{\mathbb{R}}^n}\frac{f(y)}{|x-y|^{n-\alpha}}dy,
	\end {align*}
where $\gamma(\alpha):={\pi}^{\frac{n}{2}}2^{\alpha}\Gamma(\frac{\alpha}{2})/\Gamma(\frac{n-\alpha}{2})$. Riesz potentials have important applications in fields such as signal processing, partial differential equations, and probability theory, and have recently been used in image encryption \cite{FLY2024}.

	Due to the different definitions of these two types of operators, they play very different and important roles in image processing. By comparing the above two operators, apart from the coefficients, we can easily observe two differences.
	\begin{enumerate}[(i)] 
	\item The Riesz transform is a principal value integral, while the Riesz potential is not. 
  \item   The Riesz potential involves an additional parameter $\alpha$ compared to the Riesz transform.
	\end{enumerate}

	Yabuta \cite{Y1985} introduced non-convolution $\omega$-type Calder\'{o}n-Zygmund operators which are more general than the Riesz transform. In 2014, Lu and Zhang \cite{LZ2014} introduced multilinear Calder\'{o}n-Zygmund operators with Dini's type kernels and established the boundedness of these operators and multilinear commutators on weighted Lebesgue spaces. In 2016, Zhang and Sun \cite{ZS2016} established the boundedness of multilinear  iterated commutators generated by multilinear Calder\'{o}n-Zygmund operators with Dini's type kernels and {\rm BMO} functions on weighted Lebesgue spaces.
	
	 Recently, Wu and Zhang \cite{WZ2023} proposed multilinear fractional integral operators with Dini's type kernels. They established the boundedness of these operators on weighted Lebesgue spaces and variable exponent Lebesgue spaces, whose weight functions belong to $A_{\vec P,q}$.  Suppose that $\omega\left(t\right): [0,\infty)\to[0,\infty)$ and $\omega\left(t\right)$ is a non-negative, non-decreasing function with $0<\omega\left(1\right)<\infty$. For any $a>0$,
\begin{enumerate}[(i)] 
\item  $\omega$ is said to satisfy the Dini($a$) condition, denoted as $\omega\in {\rm Dini}(a)$, if
\begin{align*}
	|\omega|_{{\rm Dini}(a)}:=\int_0^1\omega^a\left(t\right)\frac{dt}{t}<\infty.
	\end {align*}
\item $\omega$ is said to satisfy the ${\log}^m$-${\rm Dini}(a)$ condition if the following inequality holds
\begin{align*}
	\int_0^1\omega^a\left(t\right)\left(1+\log t^{-1}\right)^m\frac{dt}t<\infty,
	\end {align*}
where $m\in\mathbb{Z}^{+}:=\{1, 2, 3 ,4,\ldots\}$.
\end{enumerate}

  It is straightforward to verify that the ${\log}^m$-${\rm Dini}(a)$ conditon is more stronger than the Dini($a$) condition, and if $0<a_1<a_2<\infty$, then Dini($a_1$) $\subset$ Dini($a_2$).
In particular, if $\omega\in {\rm Dini}(1)$, then 
\begin{align*}
	\sum_{j=0}^\infty\omega\left(2^{-j}\right)\sim\int_0^1\omega\left(t\right)\frac{dt}t<\infty.
	\end {align*}
And if $\omega\in\log$-${\rm Dini}(1)$, meaning
\begin{align*}
	\int_0^1\omega\left(t\right)\left(1+\log t^{-1}\right)\frac{dt}t<\infty,
	\end {align*}
then $\omega\in {\rm Dini}(1)$ and
\begin{align*}
	\sum_{k=1}^\infty k\omega\left(2^{-k}\right)\sim\int_0^1\omega\left(t\right)\left(1+\log t^{-1}\right)\frac{dt}t<\infty.
	\end {align*}

\begin{definition}\label{definition1.1}{\rm(see \cite{WZ2023})}
\rm  Let $0<\alpha<mn$ and $K_{\alpha}\left(x,y_1,\ldots,y_m\right)$ be a locally integrable function defined away from the diagnal $x=y_1=\cdots=y_m$ in $\left({\mathbb{R}}^n\right)^{m+1}$, which is called the $m$-linear fractional integral kernel of type $\omega(t)$, if for some $A>0$, it satisfies the following size condition
	\begin{equation}\label {1.1}
	|K_{\alpha}\left(x,\vec{y}\right)|\leq\frac{A}{\left(\sum_{j=1}^m|x-y_j|\right)^{mn-\alpha}},
	\end{equation}
and the smoothness condition
\begin{equation}\label {1.2}
\left|K_{\alpha}\left(x,\vec{y}\right)-K_{\alpha}\left(x',\vec{y}\right)\right|\leq\frac{A}{\left(\sum_{j=1}^m|x-y_j|\right)^{mn-\alpha}}\omega\left(\frac{|x-x'|}{\sum_{j=1}^m|x-y_j|}\right),
\end{equation}
where $|x-x'|\leq\frac{1}{2}\max_{1\leq j\leq m}|x-y_j|$, and 
\begin{equation}\label {1.3}
\begin{split}
&\left|K_{\alpha}\left(x,y_1,\ldots,y_j,\ldots,y_m\right)-K_{\alpha}\left(x,y_1,\ldots,y_j',\ldots,y_m\right)\right|\\
\leq&\frac{A}{\left(\sum_{j=1}^m|x-y_j|\right)^{mn-\alpha}}\omega\left(\frac{|y_j-y_j'|}{\sum_{j=1}^m|x-y_j|}\right),
\end{split}
\end{equation}
where $|y_j-y_j'|\leq\frac{1}{2}\max_{1\leq j\leq m}|x-y_j|$.

 We say $T_{\alpha}:\mathscr{S}\left({\mathbb{R}}^n\right)\times\cdots\times \mathscr{S}\left({\mathbb{R}}^n\right)\rightarrow \mathscr{S}'\left({\mathbb{R}}^n\right)$ an $m$-linear fractional integral operator with kernel of type $\omega(t)$, $K_{\alpha}\left(x,\vec y\right)$, if
\begin{equation} \label {1.4}
T_{\alpha}(\vec f)\left(x\right):=\int_{\left(\mathbb{R}^n\right)^m}K_{\alpha}\left(x,\vec y\right)\prod_{j=1}^{m}f_j\left(y_j\right)d\vec y,
\end{equation}
where $x\notin\cap_{j=1}^{m}{\rm supp}f_j$ and each $f_j\in \mathscr{S}\left({\mathbb{R}}^n\right)$, $j=1,\ldots,m$.
\end{definition}

  Lin and Xiao \cite{LX2017} introduced a kind of multilinear singular integral operators with generalized kernels and they established the weighted norm inequalities on the product of weighted Lebesgue spaces. Inspired by \cite{LX2017} and \cite{WZ2023}, we focus on a class of generalized fractional integral kernels by attenuting the condition \eqref{1.2}. For $k\in\mathbb{Z}^{+}$,
\begin{equation} \label {1.5}
\begin{aligned}
 &\left(\int_{Q(x,2^{k+2}\sqrt{mn}|x-x'|)^m \setminus Q(x,2^{k+1}\sqrt{mn}|x-x'|)^m}|K_{\alpha}(x,\vec y)-K_{\alpha}(x',\vec y)|^{p_{0}}d\vec y\right)^{\frac{1}{p_0}}\\
\leq &CC_k2^{k(\alpha-\frac{mn}{p_0'})}|x-x'|^{\alpha-\frac{mn}{p_0'}},
 \end{aligned}
\end{equation}
where $Q\left(x,2^{k+2}\sqrt{mn}|x-x'|\right)$denotes the cube centered at $x$ with the sidelength $2^{k+2}\sqrt{mn}|x-x'|$, $\left(p_0,p_0'\right)$ is a pair of positve numbers satisfying $\frac{1}{p_0}+\frac{1}{p_0'}=1$, $1<p_0<\infty$ and $C_k$ is a positve constant depending on $k$.
  	
\begin{definition}\label{definition1.2}
\rm Let $0<\alpha p_0'<mn$ and $T_{\alpha}$ be an $m$-linear fractional integral operator defined by \eqref{1.4}. Then $T_{\alpha}$ is called an $m$-linear fractional integral operator with generalized kernel if the following conditions are satisfied.
\begin{enumerate}[(1)] 
\item The kernel function satisfies the size conditon \eqref{1.1} and the smoothness conditon \eqref{1.5}.
\item  For fixed $1\leq s_1,\ldots,s_m\leq p_0'$ with $\frac{1}{s}=\frac{1}{s_1}+\cdots+\frac{1}{s_m}$, $T_\alpha$ maps $L^{s_{1}}\times\cdots\times L^{s_{m}}\rightarrow L^{\frac{sn}{n-s\alpha},\infty}$.
\end{enumerate}
\end{definition}

\begin{remark}\label{remark 1.1}
\rm As a matter of fact, when $C_k=\omega\left(2^{-k}\right)$, it is straightforward to deduce that the condition \eqref{1.2} implies the condition \eqref{1.5} for any $1<p_0<\infty$. Denote $\varOmega_k:=Q\left(x,2^k\sqrt{mn}|x-x'|\right)$.
\begin{align*}
&\left(\int_{Q(x,2^{k+2}\sqrt{mn}|x-x'|)^m \setminus Q(x,2^{k+1}\sqrt{mn}|x-x'|)^m}|K_{\alpha}(x,\vec y)-K_{\alpha}(x',\vec y)|^{p_{0}}d\vec y\right)^{\frac{1}{p_0}}\\
\leq&\left(\int_{(\varOmega_{k+2})^m\setminus (\varOmega_{k+1})^m}\left[\frac{A}{\left(\sum_{j=1}^{\infty}|x-y_j|\right)^{mn-\alpha}}\omega\left(\frac{|x-x'|}{\sum_{j=1}^{\infty}|x-y_j|}\right)\right]^{p_0}d\vec{y}\right)^{\frac{1}{p_0}}\\
\leq&\left(\int_{(\varOmega_{k+2})^m\setminus (\varOmega_{k+1})^m}\left[\frac{A}{\left(\max\limits_{1\leq j\leq m}|x-y_j|\right)^{mn-\alpha}}\omega\left(\frac{|x-x'|}{\max\limits_{1\leq j\leq m}|x-y_j|}\right)\right]^{p_0}d\vec{y}\right)^{\frac{1}{p_0}}\\
\leq&\left(\int_{(\varOmega_{k+2})^m\setminus (\varOmega_{k+1})^m}\left[\frac{A}{\left(2^k|x-x'|\right)^{mn-\alpha}}\omega\left(2^{-k}\right)\right]^{p_0}d\vec{y}\right)^{\frac{1}{p_0}}\\
\leq&C\omega\left(2^{-k}\right)2^{-k\left(mn-\alpha\right)}|x-x'|^{-\left(mn-\alpha\right)}\left(\int_{(\varOmega_{k+2})^m}d\vec y\right)^{\frac{1}{p_0}}\\
=& C\omega\left(2^{-k}\right)2^{-k\left(mn-\alpha\right)}|x-x'|^{-\left(mn-\alpha\right)}\left(2^{k+2}\sqrt{mn}|x-x'|\right)^{\frac{mn}{p_0}}\\
\leq &C\omega\left(2^{-k}\right)2^{k(\alpha-\frac{mn}{p_0'})}|x-x'|^{\alpha-\frac{mn}{p_0'}}\\
\leq &CC_k2^{k(\alpha-\frac{mn}{p_0'})}|x-x'|^{\alpha-\frac{mn}{p_0'}}.
\end{align*}
\end{remark}
\begin{remark}\label{remark 1.2}
\rm when $\alpha=0$, the $m$-linear fractional integral operator with generalized kernel corresponds precisely to the multilinear Calder\'{o}n-Zygmund operator with generalized kernel studied by Yang, et al \rm\cite{YLL2024}, as well as Gao, et al \rm\cite{GLY2024}.
\end{remark}
The article is structured as follows. In Section 2, we establish the boundedness of multilinear fractional integral operators with generalized kernels and multilinear commutators on weighted Lebesgue spaces. Section 3 focuses on demonstrating the boundedness of multilinear fractional integral operators and multilinear commutators on variable exponent Lebesgue spaces.

Throughout this article, the letter $C$ always denotes a constant independent of the main paramaters involved, whose value may vary from line to line. A cube $Q\subset {\mathbb{R}}^n$ always refers to a cube with sides parallel to the coordinate axes, and we denote its side length by $l\left(Q\right)$. For $t>0$, the notation $tQ$ represents the cube with the same center as $Q$ and side length $l\left(tQ\right)=tl\left(Q\right)$. Denote by $|S|$ the Lebesgue measure and $\chi_{S}$ represents the characteristic function for a measurable set $S\subset {\mathbb{R}}^n$. $B\left(x,r\right)$ denotes the ball centered at $x$ with ridius $r$. The notation $X\sim Y$  indicates that there exists a constant $C>0$ such that $C^{-1}Y\leq X\leq CY$. For any index $1<q\left(x\right)<\infty$, we denote its conjugate index by $q'\left(x\right):=\frac{q\left(x\right)}{q\left(x\right)-1}$. Occasionally, we use the notation $\vec f=\left(f_1,\ldots,f_m\right)$, $T(\vec f):=T\left(f_1,\ldots,f_m\right)$, $d\vec y:=dy_1\cdots dy_m$ and $\left(x,\vec y\right):=\left(x,y_1,\ldots,y_m\right)$ for convenience. For a set $E$ and a positive integer $m$, we use the notation $\left(E\right)^m:=\underbrace{E\times\cdots\times E}_{m}$ sometimes.

\section{The Boundedness on Weighted Lebesgue Spaces}\label{sec2}

\quad\quad The arrangement of this Section is as follows. In Subsection 2.1, we present certain definitions and symbols that will be utilized later on. Subsection 2.2 is dedicated to establishing the pointwise estimate for the sharp maximal function. Subsection 2.3 addresses the boundedness of multilinear fractional integral operators with generalized kernels and multilinear commutators on weighted Lebesgue spaces.

\subsection{Definitions and Lemmas}
\begin{definition}\label{definition2.1}{\rm(see \cite{LDY2007})}
\rm  Let $f\in L_{loc}^{1}({\mathbb{R}}^n)$ and denote by $M$ the usual Hardy-Littlewood maximal operator. For a cube $Q\subset{\mathbb{R}}^n$ and $\delta>0$, the maximal function $M_{\delta}$ is defined by
\begin{align*}
M_{\delta}\left(f\right)\left(x\right):=\left[M\left(|f|^{\delta}\right)\left(x\right)\right]^{\frac{1}{\delta}}=\left(\sup\limits_{Q\ni x }\frac{1}{|Q|}\int_Q|f\left(y\right)|^{\delta}dy\right)^{\frac{1}{\delta}}.
\end{align*}
  Let $M^{\#}$ be the standard sharp maximal function, that is
\begin{align*}
M^{\#}f\left(x\right):=\sup\limits_{Q\ni x }\frac{1}{|Q|}\int_Q|f\left(y\right)-f_Q|dy\sim\sup\limits_{Q\ni x }\inf_c\frac{1}{|Q|}\int_Q|f\left(y\right)-c|dy,
\end{align*}
where, as usual, $f_Q$ denotes the average of $f$ over the cube $Q$ containing the point $x$. The operator $M_{\delta}^{\#}$ is defined as $M_{\delta}^{\#}\left(f\right)\left(x\right):=\left[M^{\#}\left(|f|^{\delta}\right)\left(x\right)\right]^{\frac{1}{\delta}}$.
\end{definition}

\begin{definition}\label{definition2.2}
\rm For $0<\alpha<n$ and $f\in L_{loc}^{1}({\mathbb{R}}^n)$, the standard fractional maximal operator $\mathscr{M}_{\alpha}$ is defined by
\begin{align*}
\mathscr{M}_{\alpha}(f)(x):=\sup\limits_{Q\ni x }\frac{1}{|Q|^{1-\frac{\alpha}{n}}}\int_Q|f\left(y\right)|dy.
	\end{align*}
For any $r>1$, $\mathscr{M}_{\alpha,r}$ is defined by 
\begin{align*}
&\mathscr{M}_{\alpha,r}(f)(x):=\sup\limits_{Q\ni x }|Q|^{\frac{\alpha}{n}}\left(\frac{1}{|Q|}\int_Q|f\left(y\right)|^rdy\right)^{\frac{1}{r}}.
	\end{align*}
\end{definition}

\begin{definition}\label{definition2.3}{\rm(see \cite{WZ2023})}
\rm Let $0<\alpha<mn$ and the multilinear fractional maximal operator $\mathcal{M}_{\alpha}$ is defined by setting. For $\vec f=\left(f_1,\ldots,f_m\right)$ with $f_i\in L_{loc}^{1}({\mathbb{R}}^n)$, $i=1,\ldots,m$, and $x\in{\mathbb{R}}^n$,
\begin{align*}
\mathcal{M}_{\alpha}(\vec f)(x):=\sup\limits_{Q\ni x }|Q|^{\frac{\alpha}{n}}\prod_{j=1}^{m}\frac{1}{|Q|}\int_Q|f_j\left(y_j\right)|dy_j=\sup\limits_{Q\ni x }\prod_{j=1}^{m}\frac{1}{|Q|^{1-\frac{\alpha}{mn}}}\int_Q|f_j\left(y_j\right)|dy_j.
	\end{align*}
For any $r>1$, $\mathcal{M}_{\alpha,r}$ is defined by 
	\begin{align*}
&\mathcal{M}_{\alpha,r}(\vec f)(x):=\sup\limits_{Q\ni x }|Q|^{\frac{\alpha}{n}}\prod_{j=1}^{m}\left(\frac{1}{|Q|}\int_Q|f_j\left(y_j\right)|^rdy_j\right)^{\frac{1}{r}}.
	\end{align*}
\end{definition}

\begin{definition}\label{definition2.4}{\rm(see \cite{M2009})}
\rm Let $\vec P:=\left(p_1,\ldots,p_m\right)$ and $\frac{1}{p}=\frac{1}{p_1}+\cdots+\frac{1}{p_m}$ with $1\leq p_j<\infty\left(j=1,\ldots,m\right)$. Given $\vec\omega:=\left(\omega_1,\ldots,\omega_m\right)$ with nonegative functions $\omega_1,\ldots,\omega_m$ on ${\mathbb{R}}^n$ and $p\leq q<\infty$, we say that $\vec\omega$ satisfies the $A_{\vec P,q}$ condition if

\begin{align*}
\sup_{Q}\left(\frac{1}{|Q|}\int_{Q}v_{\vec{w}}\left(x\right)^{q}dx\right)^{\frac{1}{q}}\prod_{j=1}^{m}\left(\frac{1}{|Q|}\int_{Q}\omega_j\left(x\right)^{-p_j'}dx\right)^{\frac{1}{p_j'}}<\infty,
\end{align*}
where the supremum is taken over all cubes $Q\subset{\mathbb{R}}^n$, $v_{\vec\omega}=\prod_{j=1}^{m}\omega_j$, and $\left(\frac{1}{|Q|}\int_{Q}\omega_j\left(x\right)^{-p_j'}dx\right)^{\frac{1}{p_j'}}$ in the case $p_j=1$ is understood as $\left(\inf_Q\omega_j\right)^{-1}$ when $p_j=1$.
\end{definition}

\begin{definition}\label{definition2.5}{\rm(see \cite{WYL2023})}
\rm Suppose $\vec b:=\left(b_1,\ldots,b_m\right)$ with $b_j\in L^1_{loc}({\mathbb{R}}^n)(j=1,\ldots,m)$. We define the $m$-linear commutator $T_{\vec b,\alpha}$ to be
\begin{align*}
T_{\vec b,\alpha}(\vec f)\left(x\right):=\sum_{i=1}^mT_{\vec b,\alpha}^i(\vec f)\left(x\right).
\end{align*}

\noindent Each term represents the commutator of $T_{\alpha}$ in the $i$-th entry with $b_i$, which is defined by
\begin{align*}
T_{\vec b,\alpha}^i(\vec f)\left(x\right):=b_i\left(x\right)T_{\alpha}\left(f_1,\ldots,f_i,\ldots,f_m\right)\left(x\right)-T_{\alpha}\left(f_1,\ldots,b_if_i,\ldots,f_m\right)\left(x\right).
\end{align*}
\end{definition}

\begin{definition}\label{definition2.6}{\rm(see \cite{L1995})} 
\rm Suppose that $b\in L_{loc}^{1}({\mathbb{R}}^n)$. Let
\begin{align*}
\|b\|_{{\rm BMO}}:=\sup_Q\frac{1}{|Q|}\int_Q|b\left(x\right)-b_Q|dx,
\end{align*}
where the supreme is taken over all cube $Q\subset{\mathbb{R}}^n$, and
\begin{align*}
f_Q:=\frac{1}{|Q|}\int_Qf(y)dy.
\end{align*}
Define
\begin{align*}
{\rm BMO}({\mathbb{R}}^n):=\{f:\|f\|_{\rm BMO}<\infty\}.
\end{align*}

Let $\vec b=(b_1,\ldots,b_m)$, if $b_j\in{\rm BMO}$ for $1\leq j\leq m$, then $\vec b\in {\rm BMO}^m$.
\end{definition}

\begin{lemma}\label{Lemma2.1}{\rm(see \cite{FS1972})} 
\rm Let $0<p,\delta<\infty$ and $\omega$ be any $A_{\infty}$-weight ($A_{\infty}:=\cup_{p\geq 1}A_p$). Then there exists a constant $C>0$ (depending on the $A_{\infty}$ constant of $\omega$), such that the inequalities 
	
	\begin{align*}
	\int\left(M_{\delta}f\left(x\right)\right)^p\omega\left(x\right)dx\leq C\int(M_{\delta}^{\#}f\left(x\right))^p\omega\left(x\right)dx,
	\end{align*}
and 
	\begin{align*}
\left\|M_{\delta}\left(f\right)\right\|_{L^{p,\infty}\left(\omega\right)}\leq C\left\|M_{\delta}^{\#}\left(f\right)\right\|_{L^{p,\infty}\left(\omega\right)},
	\end{align*}
hold for any function $f$ for which the left hand side is finite.
	\end{lemma}

\begin{lemma}\label{Lemma2.2}{\rm(see \cite{LZ2014})} 
\rm Let $0<p<u<\infty$. There exists a positive constant $C=C_{p,u}$ such that the following inequality holds
\begin{align*}
	|Q|^{-\frac{1}{p}}\left\|f\right\|_{L^p\left(Q\right)}\leq C|Q|^{-\frac{1}{u}}\left\|f\right\|_{L^{u,\infty}\left(Q\right)}.
	\end{align*}
	\end {lemma}

\begin{lemma}\label{Lemma2.3}{\rm(see \cite{M2009})}
\rm Let $0<\alpha<mn$, $\frac{1}{p}=\frac{1}{p_1}+\cdots+\frac{1}{p_m}$ with $1\leq p_j<\infty\left(j=1,\ldots,m\right)$, and $\frac{1}{q}=\frac{1}{p}-\frac{\alpha}{n}>0$.

\begin{enumerate}[(i)] 
	\item If $1<p_j<{\infty}$ for all $j=1,\ldots,m$, then for $\vec\omega\in A_{\vec P,q}$, there exists a constant $C>0$ independent of $\vec f$ such that
\begin{align*}
\left\|\mathcal{M}_{\alpha}(\vec f)\right\|_{L^q\left(v_{\omega}^q\right)}\leq C\prod_{j=1}^m\left\|f_j\right\|_{L^{p_j}\left(\omega_j^{p_j}\right)}.
\end{align*}
	
	\item If $1\leq p_j<{\infty}$ for all $j=1,\ldots,m$ and at least one $p_j=1$ for some $j=1,\ldots,m$, then for $\vec\omega\in A_{\vec P,q}$, there exists a constant $C>0$ independent of $\vec f$ such that
\begin{align*}
\left\|\mathcal{M}_{\alpha}(\vec f)\right\|_{L^{q,\infty}\left(v_{\omega}^q\right)}\leq C\prod_{j=1}^m\left\|f_j\right\|_{L^{p_j}\left(\omega_j^{p_j}\right)}.
\end{align*}
	\end{enumerate}
\end{lemma}

\begin{lemma}\label{Lemma2.4} {\rm(see \cite{M2009})}
\rm Let $0<\alpha<mn$, $\vec\omega=\left(\omega_1,\ldots,\omega_m\right)$, $v_{\vec\omega}=\prod_{j=1}^{m}\omega_j$, $\vec P=\left(p_1,\ldots,p_m\right)$, $\frac{1}{p}=\frac{1}{p_1}+\cdots+\frac{1}{p_m}$ with $1\leq p_j<\infty\left(j=1,\ldots,m\right)$, and $\frac{1}{q}=\frac{1}{p}-\frac{\alpha}{n}>0$. Suppose $\vec\omega\in A_{\vec P,q}$, then
\[
\begin{cases}
\omega_j^{-p_j'}\in A_{mp_j'}\left(j=1,\ldots,m\right),\\
v_{\vec\omega}^q\in A_{mq}.
\end{cases}
\]
	\end {lemma}

\begin{lemma}\label{Lemma2.5}{\rm(see \cite{BHS2011})} 
\rm Let $a\geq 1$. If $b\in {\rm BMO}$, then for all cubes $Q$, we have
\begin{itemize}
\item [\rm (i)]$\left(\frac{1}{|Q|}\int_Q|b\left(y\right)-b_Q|dy\right)^{\frac{1}{a}}\leq C\|b\|_{{\rm BMO}}$.
\item [\rm (ii)]$\left(\frac{1}{|2^kQ|}\int_{2^kQ}|b\left(y\right)-b_Q|dy\right)^{\frac{1}{a}}\leq Ck\|b\|_{{\rm BMO}}$, for $k\in\mathbb{N}^+$.
\end{itemize}
\end{lemma}

\subsection{Sharp Maximal Pointwise Estimates}
\quad\quad In this subsection, we initially obtain the sharp maximal estimates for multilinear fractional integral operators with generalized kernels and multilinear commutators generated by ${\rm BMO}$ functions and these operators.
\begin{theorem}\label{Theorem 2.1}
\rm Let $m\geq2$, $0<\alpha p_0'<mn$, $T_{\alpha}$ be an $m$-linear fractional integral operator with generalized kernel as  defined in Definition \ref{definition1.2} and $\sum_{k=1}^{\infty}C_k<\infty$. Assuming $0<\delta<\min\left\{1,\frac{sn}{n-s\alpha}\right\}$, there exists a constant $C>0$ such that for all $m$-tuples of bounded measurable functions $\vec f=\left(f_1,\ldots,f_m\right)$ with compact support, the following holds
\begin{align*}
	M_{\delta}^{\#}(T_{\alpha}(\vec f))\left(x\right)\leq C\mathcal{M}_{\alpha,p_0'}(\vec f)\left(x\right).
	\end{align*}
	\end{theorem}
\begin{proof}
Since $||a|^e-|b|^e|\leq|a-b|^e$ for $0<e<1$, for any $Q\ni x$, we can estimate 
\begin{align*}
\left(\frac{1}{|Q|}\int_Q\left|\left|T_{\alpha}(\vec f)\left(z\right)\right|^{\delta}-|c|^{\delta}\right|dz\right)^{\frac{1}{\delta}}\leq C\left(\frac{1}{|Q|}\int_Q\left|T_{\alpha}(\vec f)\left(z\right)-c\right|^{\delta}dz\right)^{\frac{1}{\delta}}.
	\end{align*}

  Let $Q^{\ast}:=14n\sqrt{mn}Q$, we decomepose $f_j:=f_j^0+f_j^{\infty}$ with $f_j^0:=f_j\chi_{Q^{\ast}}$, then
\begin{align*}
\prod_{j=1}^mf_j\left(y_j\right)&=\prod_{j=1}^m\left(f_j^0\left(y_j\right)+f_j^{\infty}\left(y_j\right)\right)\\
&=\sum_{\left(\rho_1\ldots,\rho_m\right)\in\left\{0,\infty\right\}}f_1^{\rho_1}\left(y_1\right)\times\cdots\times f_m^{\rho_m}\left(y_m\right)\\
&=\prod_{j=1}^mf_j^0\left(y_j\right)+\sum_{\left(\rho_1,\ldots,\rho_m\right)\in\rho}f_1^{\rho_1}\left(y_1\right)\times\cdots\times f_m^{\rho_m}\left(y_m\right),
\end{align*}
where $\rho=\{\left(\rho_1,\ldots,\rho_m\right)$: there is at least one $\rho_j={\infty}\}$. It is easy to see that
\begin{align*}
	T_{\alpha}(\vec f)\left(z\right)=T_{\alpha}\left(f_1^0,\ldots,f_m^0\right)\left(z\right)+\sum_{\left(\rho_1,\ldots,\rho_m\right)\in\rho}T_{\alpha}\left(f_1^{\rho_1},\ldots,f_m^{\rho_m}\right)\left(z\right).
	\end{align*}

  Let $c=\sum_{(\rho_1,\ldots,\rho_m)\in \rho}c_{\rho_1,\ldots,\rho_m}$, then we can get
\begin{align*}
&\left(\frac{1}{|Q|}\int_Q\left|T_{\alpha}(\vec f)\left(z\right)-c\right|^{\delta}dz\right)^{\frac{1}{\delta}}\\
\leq &C\left(\frac{1}{|Q|}\int_Q\left|T_{\alpha}\left(f_1^0,\ldots,f_m^0\right)\left(z\right)\right|^{\delta}dz\right)^{\frac{1}{\delta}}\\
&+C\sum_{\left(\rho_1,\ldots,\rho_m\right)\in\rho}\left(\frac{1}{|Q|}\int_Q\left|T_{\alpha}\left(f_1^{\rho_1},\ldots,f_m^{\rho_m}\right)\left(z\right)-c_{\rho_1,\ldots,\rho_m}\right|^{\delta}dz\right)^{\frac{1}{\delta}}\\
:=&\uppercase\expandafter{\romannumeral+1}+\sum_{(\rho_1,\ldots,\rho_m)\in \rho}\uppercase\expandafter{\romannumeral+2}_{\rho_1,\ldots,\rho_m}.
\end{align*}

  From Definition \ref{definition1.2} (2), H\"{o}lder's inequality and Lemma \ref{Lemma2.2}, it follows that
\begin{align*}
\uppercase\expandafter{\romannumeral+1}&\leq C|Q|^{-\frac{n-s\alpha}{sn}}\|T_{\alpha}\left(f_1^0,\ldots,f_m^0\right)\|_{{L^{\frac{sn}{n-s\alpha},\infty}}\left(Q\right)}\\
&\leq C|Q|^{\frac{\alpha}{n}}|Q|^{-\frac{1}{s}}\left(\int_{Q^{\ast}}|f_1\left(y_1\right)|^{s_1}dy_1\right)^{\frac{1}{s_1}}\times\cdots\times\left(\int_{Q^{\ast}}|f_m\left(y_m\right)|^{s_m}dy_m\right)^{\frac{1}{s_m}}\\
&\leq C|Q^{\ast}|^{\frac{\alpha}{n}}\left(\frac{1}{|Q^{\ast}|}\int_{Q^{\ast}}|f_1\left(y_1\right)|^{s_1}dy_1\right)^{\frac{1}{s_1}}\times\cdots\times\left(\frac{1}{|Q^{\ast}|}\int_{Q^{\ast}}|f_m\left(y_m\right)|^{s_m}dy_m\right)^{\frac{1}{s_m}}\\
&\leq C|Q^{\ast}|^{\frac{\alpha}{n}}\prod_{j=1}^m\left(\frac{1}{|Q^{\ast}|}\int_{Q^{\ast}}|f_j\left(y_j\right)|^{p_0'}dy_m\right)^{\frac{1}{p_0'}}\\
&\leq C\mathcal{M}_{\alpha,p_0'}(\vec f)\left(x\right).
\end{align*}

  To estimate $\uppercase\expandafter{\romannumeral+2}_{\rho_1,\ldots,\rho_m}$, we choose $c_{\rho_1,\ldots,\rho_m}:=T_{\alpha}\left(f_1^{\rho_1},\ldots,f_m^{\rho_m}\right)\left(z_0\right)$ for a fixed $z_0\in 4Q \setminus 3Q$. Let $\varOmega_k:=Q\left(z_0,2^k\sqrt{mn}|z-z_0|\right)$, $k\in\mathbb{N}^+$. Note that $\varOmega_{2}\subset Q^{\ast}$, so $({\mathbb{R}}^n)^m \setminus (Q^{\ast})^m \subset ({\mathbb{R}}^n)^m \setminus (\varOmega_2)^m$. When $z\in Q$, we obtain $l\left(Q\right)\leq|z-z_0|\leq\frac{5}{2}\sqrt{n}l\left(Q\right)$.  By H\"{o}lder's inequality, we conclude that
\begin{align*}
\uppercase\expandafter{\romannumeral+2}_{\rho_1,\ldots,\rho_m}&\leq \frac{C}{|Q|}\int_Q\left|T_{\alpha}\left(f_1^{\rho_1},\ldots,f_m^{\rho_m}\right)\left(z\right)-T_{\alpha}\left(f_1^{\rho_1},\ldots,f_m^{\rho_m}\right)\left(z_0\right)\right|dz\\ 
&\leq \frac{C}{|Q|}\int_Q\int_{\left({\mathbb{R}}^n\right)^m\setminus \left(Q^{\ast}\right)^m}|K_{\alpha}\left(z,\vec y\right)-K_{\alpha}\left(z_0,\vec y\right)||\prod_{j=1}^m|f_j\left(y_j\right)|d{\vec y}dz\\
&\leq \frac{C}{|Q|}\int_Q\int_{\left({\mathbb{R}}^n\right)^m\setminus \left(\varOmega_2\right)^m}|K_{\alpha}\left(z,\vec y\right)-K_{\alpha}\left(z_0,\vec y\right)||\prod_{j=1}^m|f_j\left(y_j\right)|d{\vec y}dz\\
&\leq \frac{C}{|Q|}\int_Q\sum_{k=1}^{\infty}\int_{(\varOmega_{k+2})^m\setminus (\varOmega_{k+1})^m}|K_{\alpha}\left(z,\vec y\right)-K_{\alpha}\left(z_0,\vec y\right)||\prod_{j=1}^m|f_j\left(y_j\right)|d{\vec y}dz\\
&\leq \frac{C}{|Q|}\int_Q\sum_{k=1}^{\infty}\left(\int_{(\varOmega_{k+2})^m\setminus (\varOmega_{k+1})^m}|K_{\alpha}\left(z,\vec y\right)-K_{\alpha}\left(z_0,\vec y\right)|^{p_0}d{\vec y}\right)^{\frac{1}{p_0}}\\
&\ \ \ \ \times\left(\int_{\left(\varOmega_{k+2}\right)^m}\prod_{j=1}^m|f_j\left(y_j\right)|^{p_0'}d{\vec y}\right)^{\frac{1}{p_0'}}dz\\
&\leq \frac{C}{|Q|}\int_Q\sum_{k=1}^{\infty}C_k|z-z_0|^{\alpha-\frac{mn}{p_0'}}2^{k(\alpha-\frac{mn}{p_0'})}|\varOmega_{k+2}|^{\frac{m}{p_0'}-\frac{\alpha}{n}}\\
&\ \ \ \ \times|\varOmega_{k+2}|^{\frac{\alpha}{n}}\prod_{j=1}^{m}\left(\frac{1}{|\varOmega_{k+2}|}\int_{\varOmega_{k+2}}|f_j\left(y_j\right)|^{p_0'}dy_j\right)^{\frac{1}{p_0'}}dz\\
&\leq C\mathcal{M}_{\alpha,p_0'}(\vec f)\left(x\right).
\end{align*}
Combining the above estimates we get the desired result.
\end{proof}

\begin{theorem}\label{Theorem2.2} 
\rm Let $m\geq2$, $0<\alpha p_0'<mn$, $T_{\alpha}$ be an $m$-linear fractional integral operator with generalized kernel as  defined in Definition \ref{definition1.2} and $\sum_{k=1}^{\infty}kC_k<\infty$. Assuming $0<\delta<\min\left\{1,\frac{sn}{n-s\alpha}\right\}$, $p_0'<t<\infty$, $\delta<\varepsilon<\infty$ and $\vec b\in {\rm BMO}^m$, there exists a constant $C>0$ such that for all bounded measurable functions with compact support $m$-tuples $\vec f=\left(f_1,\ldots,f_m\right)$, we have
\begin{align*}
M_{\delta}^{\#}(T_{\vec b,\alpha}(\vec f))\left(x\right)\leq C\big\|\vec b\big\|_{\left({\rm BMO}\right)^m}\left(\mathcal{M}_{\alpha,t}(\vec f)\left(x\right)+M_{\varepsilon}(T_{\alpha}(\vec f))\left(x\right)\right),
\end{align*}
where $\big\|\vec b\big\|_{\left({\rm BMO}\right)^m}:=\max_{1\leq i\leq m}\|b_i\|_{{\rm BMO}}$.
\end {theorem}

\begin{proof}
Let $Q^{\ast}:=14n\sqrt{mn}Q$, we decomepose $f_j:=f_j^0+f_j^{\infty}$ with $f_j^0:=f_j\chi_{Q^{\ast}}$, then
\begin{align*}
\prod_{j=1}^mf_j\left(y_j\right)&=\prod_{j=1}^m\left(f_j^0\left(y_j\right)+f_j^{\infty}\left(y_j\right)\right)\\
&=\sum_{\left(\rho_1,\ldots,\rho_m\right)\in\left\{0,\infty\right\}}f_1^{\rho_1}\left(y_1\right)\times\cdots\times f_m^{\rho_m}\left(y_m\right)\\
&=\prod_{j=1}^mf_j^0\left(y_j\right)+\sum_{\left(\rho_1,\ldots,\rho_m\right)\in\rho}f_1^{\rho_1}\left(y_1\right)\times\cdots\times f_m^{\rho_m}\left(y_m\right),
\end{align*}
where $\rho=\{\left(\rho_1,\ldots,\rho_m\right)$: there is at least one $\rho_j={\infty}\}$. It is easy to see that
\begin{align*}
&T_{\alpha}\left(f_1,\ldots,\left(b_j-\lambda\right)f_j,\ldots,f_m\right)\left(z\right)\\
=&T_{\alpha}\left(f_1^0,\ldots,\left(b_j-\lambda\right)f_j^0,\ldots,f_m^0\right)\left(z\right)\\
\ \ \ \ &+\sum_{\left(\rho_1,\ldots,\rho_m\right)\in\rho}T_{\alpha}\left(f_1^{\rho_1},\ldots,\left(b_j-\lambda\right)f_j^{\rho_j},\ldots,f_m^{\rho_m}\right)\left(z\right).
	\end{align*}

  Fix $x\in{\mathbb{R}}^n$ and let $Q$ be a cube containing $x$. Denote $\lambda={b_j}_{Q^{\ast}}$. Let $c_j=-\sum_{(\rho_1,\ldots,\rho_m)\in \rho}c_{j,\rho_1,\ldots,\rho_m}$, then
\begin{align*}
&\left(\frac{1}{|Q|}\int_Q|T_{\vec b,\alpha}^j(\vec f)(z)-c_j|^{\delta}dz\right)^{\frac{1}{\delta}}\\
\leq &C\left(\frac{1}{|Q|}\int_Q|(b_j(z)-\lambda)T_{\alpha}(\vec f)(z)|^{\delta}dz\right)^{\frac{1}{\delta}}+C\left(\frac{1}{|Q|}\int_Q|T_{\alpha}(f_1,\ldots,(b_j-\lambda)f_j,\ldots,f_m)(z)+c_j|^{\delta}dz\right)^{\frac{1}{\delta}}\\
\leq &C\left(\frac{1}{|Q|}\int_Q|(b_j(z)-\lambda)T_{\alpha}(\vec f)(z)|^{\delta}dz\right)^{\frac{1}{\delta}}+C\left(\frac{1}{|Q|}\int_Q|T_{\alpha}(f_1^0,\ldots,(b_j-\lambda)f_j^0,\ldots,f_m^0)(z)|^{\delta}dz\right)^{\frac{1}{\delta}}\\
\ \ \ \ &+C\sum_{(\rho_1,\ldots,\rho_m)\in\rho}\left(\frac{1}{|Q|}\int_Q|T_{\alpha}(f_1^{\rho_1},\ldots,(b_j-\lambda)f_j^{\rho_j},\ldots,f_m^{\rho_m})(z)-c_{j,\rho_1,\ldots,\rho_m}|^{\delta}dz\right)^{\frac{1}{\delta}}\\
:=&\uppercase\expandafter{\romannumeral+3}+\uppercase\expandafter{\romannumeral+4}+\sum_{(\rho_1,\ldots,\rho_m)\in \rho}\uppercase\expandafter{\romannumeral+5}_{\rho_1,\ldots,\rho_m}
	\end{align*}

  Choosing $1<h_1'<\min\{\frac{1}{1-\delta},\frac{\varepsilon}{\delta}\}$ with $\frac{1}{h_1'}+\frac{1}{h_1}=1$, by H\"{o}lder's inequality, we conclude that
\begin{align*}
\uppercase\expandafter{\romannumeral+3}&\leq C\left(\frac{1}{|Q|}\int_Q|b_j(z)-\lambda|^{\delta h_1}dz\right)^{\frac{1}{\delta h_1}}\left(\frac{1}{|Q|}\int_Q|T_{\alpha}(\vec f)|^{\delta h_1'}dz\right)^{\frac{1}{\delta h_1'}}\\
&\leq C\left(\frac{1}{|Q^{\ast}|}\int_{Q^{\ast}}|b_j(z)-\lambda|^{\delta h_1}dz\right)^{\frac{1}{\delta h_1}}\left(\frac{1}{|Q|}\int_Q|T_{\alpha}(\vec f)|^{\delta h_1'}dz\right)^{\frac{1}{\delta h_1'}}\\
&\leq C\|b_j\|_{{\rm BMO}}M_{\varepsilon}(T_{\alpha}(\vec f))\left(x\right).
	\end{align*}

Let $h_2=\frac{t}{s_j}$ with $\frac{1}{h_2}+\frac{1}{h_2'}=1$. From Definition \ref{definition1.2} (2), H\"{o}lder's inequality, Lemma \ref{Lemma2.2} and $s_i\leq p_0'<t \left(i=1,\ldots,m\right)$, it follows that
 
\begin{align*}
\uppercase\expandafter{\romannumeral+4}&\leq C|Q|^{-\frac{n-s\alpha}{sn}}\left\|T_{\alpha}\left(f_1^0,\ldots,\left(b_j-\lambda\right)f_j^0,\ldots,f_m^0\right)\right\|_{L^{\frac{sn}{n-s\alpha},\infty}\left(Q\right)}\\
&\leq C|Q^{\ast}|^{\frac{\alpha}{n}}|Q^{\ast}|^{-\frac{1}{s}}\prod_{i=1,i\neq j}^m\|f_i\|_{L^{s_i}\left(Q^{\ast}\right)}\|\left(b_j-\lambda\right)f_j\|_{L^{s_j}\left(Q^{\ast}\right)}\\
&= C|Q^{\ast}|^{\frac{\alpha}{n}}\prod_{i=1,i\neq j}^m\left(\frac{1}{|Q^{\ast}|}\int_{Q^{\ast}}|f_i\left(y_i\right)|^{s_i}dy_i\right)^{\frac{1}{s_i}}\left(\frac{1}{|Q^{\ast}|}\int_{Q^{\ast}}|b_j(y_j)-\lambda|^{s_j}|f_j(y_j)|^{s_j}dy_j\right)^{\frac{1}{s_j}}\\
&\leq C|Q^{\ast}|^{\frac{\alpha}{n}}\prod_{i=1,i\neq j}^m\left(\frac{1}{|Q^{\ast}|}\int_{Q^{\ast}}|f_i\left(y_i\right)|^{s_i}dy_i\right)^{\frac{1}{s_i}}\left(\frac{1}{|Q^{\ast}|}\int_{Q^{\ast}}|b_j(y_j)-\lambda|^{s_j h_2'}dy_j\right)^{\frac{1}{s_j h_2'}}\\
&\ \ \ \ \times\left(\frac{1}{|Q^{\ast}|}\int_{Q^{\ast}}|f_j\left(y_j\right)|^{s_j h_2}dy_j\right)^{\frac{1}{s_j h_2}}\\
&\leq C\|b_j\|_{{\rm BMO}}|Q^{\ast}|^{\frac{\alpha}{n}}\prod_{i=1}^m\left(\frac{1}{|Q^{\ast}|}\int_{Q^{\ast}}|f_i\left(y_i\right)|^{t}dy_i\right)^{\frac{1}{t}}\\
&\leq C\|b_j\|_{{\rm BMO}}\mathcal{M}_{\alpha,t}(\vec f)\left(x\right).
\end{align*}

To estimate $\uppercase\expandafter{\romannumeral+5}_{\rho_1,\ldots,\rho_m}$, we choose $c_{j,\rho_1,\ldots,\rho_m}=T_{\alpha}\left(f_1^{\rho_1},\ldots,\left(b_j-\lambda\right)f_j^{\rho_j},\ldots,f_m^{\rho_m}\right)\left(z_0\right)$ for a fixed $z_0\in 4Q \setminus 3Q$. Since $\varOmega_{2}\subset Q^{\ast}$, we can get $\varOmega_{k+2}\subset 2^kQ^{\ast}$. Since $|z-z_0|\sim l(Q)$ when $z\in Q$, then $|\varOmega_{k+2}|\sim |2^kQ^{\ast}|$. Let $h_3=\frac{t}{p_0'}$ with $\frac{1}{h_3}+\frac{1}{h_3'}=1$. By H\"{o}lder's inequality, Lemma \ref{Lemma2.5} and $s_i\leq p_0'<t \left(i=1,\ldots,m\right)$, we conclude that
\begin{align*}
\uppercase\expandafter{\romannumeral+5}_{\rho_1,\ldots,\rho_m}\leq &\frac{C}{|Q|}\int_Q|T_{\alpha}\left(f_1^{\rho_1},\ldots,\left(b_j-\lambda\right)f_j^{\rho_j},\ldots,f_m^{\rho_m}\right)\left(z\right)-T_{\alpha}\left(f_1^{\rho_1},\ldots,\left(b_j-\lambda\right)f_j^{\rho_j},\ldots,f_m^{\rho_m}\right)\left(z_0\right)|dz\\
\leq &\frac{C}{|Q|}\int_Q\int_{({\mathbb{R}}^n)^m\setminus (Q^{\ast})^{m}}|K_{\alpha}\left(z,\vec y\right)-K_{\alpha}\left(z_0,\vec y\right)|\prod_{i=1}^{m}|f_i\left(y_i\right)||b_j(y_j)-\lambda|d\vec ydz\\
\leq &\frac{C}{|Q|}\int_Q\sum_{k=1}^{\infty}\int_{(\varOmega_{k+2})^m\setminus (\varOmega_{k+1})^{m}}|K_{\alpha}\left(z,\vec y\right)-K_{\alpha}\left(z_0,\vec y\right)|\prod_{i=1}^{m}|f_i\left(y_i\right)||b_j(y_j)-\lambda|d\vec ydz\\
\leq &\frac{C}{|Q|}\int_Q\sum_{k=1}^{\infty}\left(\int_{(\varOmega_{k+2})^m\setminus (\varOmega_{k+1})^{m}}|K_{\alpha}\left(z,\vec y\right)-K_{\alpha}\left(z_0,\vec y\right)|^{p_0}d\vec y\right)^{\frac{1}{p_0}}\\
\ \ \ \ &\times\left(\int_{(\varOmega_{k+2})^m}\left(\prod_{i=1}^{m}|f_i\left(y_i\right)||b_j(y_j)-\lambda|\right)^{p_0'}d\vec y\right)^{\frac{1}{p_0'}}dz\\
\leq &\frac{C}{|Q|}\int_{Q}\sum_{k=1}^{\infty}C_k|z-z_0|^{\alpha-\frac{mn}{p_0'}}2^{k(\alpha-\frac{mn}{p_0'})}|\varOmega_{k+2}|^{\frac{m}{p_0'}-\frac{\alpha}{n}}|\varOmega_{k+2}|^{\frac{\alpha}{n}}\\
\ \ \ \ &\times\prod_{i=1,i\neq j}^m\left(\frac{1}{|\varOmega_{k+2}|}\int_{\varOmega_{k+2}}|f_i\left(y_i\right)|^{p_0'}dy_i\right)^{\frac{1}{p_0'}}\left(\frac{1}{|\varOmega_{k+2}|}\int_{\varOmega_{k+2}}|b_j\left(y_j\right)-\lambda|^{p_0'}|f_j\left(y_j\right)|^{p_0'}dy_j\right)^{\frac{1}{p_0'}}dz\\
\leq &\frac{C}{|Q|}\int_{Q}\sum_{k=1}^{\infty}C_k|\varOmega_{k+2}|^{\frac{\alpha}{n}}\prod_{i=1,i\neq j}^m\left(\frac{1}{|\varOmega_{k+2}|}\int_{\varOmega_{k+2}}|f_i\left(y_i\right)|^{t}dy_i\right)^{\frac{1}{t}}\\
\ \ \ \ &\times\left(\frac{1}{|\varOmega_{k+2}|}\int_{\varOmega_{k+2}}|f_j\left(y_j\right)|^{p_0'h_3}dy_j\right)^{\frac{1}{p_0'h_3}}\left(\frac{1}{|\varOmega_{k+2}|}\int_{\varOmega_{k+2}}|b_j\left(y_j\right)-\lambda|^{p_0'h_3'}dy_j\right)^{\frac{1}{p_0'h_3'}}dz\\
\leq &\frac{C}{|Q|}\int_{Q}\sum_{k=1}^{\infty}C_k|\varOmega_{k+2}|^{\frac{\alpha}{n}}\prod_{i=1,i\neq j}^m\left(\frac{1}{|\varOmega_{k+2}|}\int_{\varOmega_{k+2}}|f_i\left(y_i\right)|^{t}dy_i\right)^{\frac{1}{t}}\\
\ \ \ \ &\times\left(\frac{1}{|\varOmega_{k+2}|}\int_{\varOmega_{k+2}}|f_j\left(y_j\right)|^{t}dy_i\right)^{\frac{1}{t}}\left(\frac{1}{|2^kQ^{\ast}|}\int_{2^kQ^{\ast}}|b_j\left(y_j\right)-\lambda|^{p_0'h_3'}dy_j\right)^{\frac{1}{p_0'h_3'}}dz\\
\leq &\frac{C}{|Q|}\int_{Q}\sum_{k=1}^{\infty}C_k|\varOmega_{k+2}|^{\frac{\alpha}{n}}\prod_{i=1}^{m}\left(\frac{1}{|\varOmega_{k+2}|}\int_{\varOmega_{k+2}}|f_i\left(y_i\right)|^{t}dy_i\right)^{\frac{1}{t}}k\|b_j\|_{{\rm BMO}}dz\\
\leq &C\|b_j\|_{{\rm BMO}}\mathcal{M}_{\alpha,t}(\vec f)\left(x\right).
\end{align*}
Then, we have
\begin{align*}
&\inf\limits_c\left(\frac{1}{|Q|}\int_Q\left|\left|T_{\vec b,\alpha}(\vec f)\left(z\right)\right|^{\delta}-c\right|dz\right)^{\frac{1}{\delta}}\leq\left(\frac{1}{|Q|}\int_Q\left|\left|\sum_{i=1}^mT_{\vec b,\alpha}^{i}(\vec f)\left(z\right)\right|^{\delta}-\left|\sum_{i=1}^mc_i\right|^{\delta}\right|dz\right)^{\frac{1}{\delta}}\\
&\leq C\sum_{i=1}^m\left(\frac{1}{|Q|}\int_Q\left|T_{\vec b,\alpha}^i(\vec f)\left(z\right)-c_i\right|^{\delta}dz\right)^{\frac{1}{\delta}}\leq C\big\|\vec b\big\|_{\left({\rm BMO}\right)^m}\left(M_{\varepsilon}(T_{\alpha}(\vec f))\left(x\right)+\mathcal{M}_{\alpha,t}(\vec f)\left(x\right)\right).
\end{align*}

\noindent  Then, we have
\begin{align*}
&M_{\delta}^{\#}(T_{\vec b,\alpha}(\vec f))\left(x\right)\thicksim\sup\limits_{Q\ni x}\inf\limits_c\left(\frac{1}{|Q|}\int_Q\left|\left|T_{\vec b,\alpha}(\vec f)\left(z\right)\right|^{\delta}-c\right|dz\right)^{\frac{1}{\delta}}\\
&\leq C\big\|\vec b\big\|_{\left({\rm BMO}\right)^m}\left(M_{\varepsilon}(T_{\alpha}(\vec f))\left(x\right)+\mathcal{M}_{\alpha,t}(\vec f)\left(x\right)\right).
\end{align*}

\noindent  Thus, Theorem \ref{Theorem2.2} is proved.
\end{proof}

\subsection{The Boundedness on Weighted Lebesgue Spaces}

\quad\quad In this subsection, we use Theorem \ref{Theorem 2.1} and Theorem \ref{Theorem2.2} to derive boundedness results for multilinear fractional integral operators with generalized kernels and multilinear commutators on weighted Lebesgue spaces. The main theorem of this subsection is presented below.

\begin{theorem}\label{Theorem 2.3} 
\rm Let $m\geq 2$, $0<\alpha p_0'<mn$, $T_{\alpha}$ be an $m$-linear fractional integral operator with generalized kernel as in Definition \ref{definition1.2} and $\sum_{k=1}^{\infty}C_k<\infty$. Suppose that $\vec P=\left(p_1,\cdots,p_m\right)$, $\frac{1}{p}=\frac{1}{p_1}+\ldots+\frac{1}{p_m}$ with $p_0'\leq p_j<\infty\left(j=1,\ldots,m\right)$, $\frac{1}{q}=\frac{1}{p}-\frac{\alpha}{n}>0$ and $\vec\omega\in A_{\frac{\vec P}{p_0'},\frac{q}{p_0'}}$.
\begin{enumerate}[(i)] 
     \item If $p_0'<p_j<\infty$ for all $j=1,\ldots,m$, then
\begin{align*}
\left\|T_{\alpha}(\vec f)\right\|_{L^{q}\left(v_{\vec\omega}^{\frac{q}{p_0'}}\right)}\leq C\prod_{j=1}^m\left\|f_j\right\|_{L^{p_j}\left(\omega_j^{\frac{p_j}{p_0'}}\right)}.
\end{align*}
     \item If $p_0'\leq p_j<\infty$ for all $j=1,\ldots,m$ and at least one of $p_j=p_0'$, then 
\begin{align*}
\left\|T_{\alpha}(\vec f)\right\|_{L^{q,\infty}\left(v_{\vec\omega}^{\frac{q}{p_0'}}\right)}\leq C\prod_{j=1}^m\left\|f_j\right\|_{L^{p_j}\left(\omega_j^{\frac{p_j}{p_0'}}\right)}.
\end{align*}
\end{enumerate}
\end{theorem}

\begin{proof}
Firstly, we choose $\delta$ such that $0<\delta<\min\left\{1,\frac{sn}{n-s\alpha}\right\}$ and let $\beta=\alpha p_0'$.
\begin{enumerate}[(i)] 
\item  By Lemma \ref{Lemma2.4}, for $\vec\omega\in A_{\frac{\vec P}{p_0'},\frac{q}{p_0'}}$ there has $v_{\vec\omega}^{\frac{q}{p_0'}}\in A_{\infty}$. Then, by Lemma \ref{Lemma2.1}, Theorem \ref{Theorem 2.1} and Lemma \ref{Lemma2.3} (i), we conclude that
\begin{align*}
\left\|T_{\alpha}(\vec f)\right\|_{L^{q}\left(v_{\vec\omega}^{\frac{q}{p_0'}}\right)}\leq &\left\|M_{\delta}(T_{\alpha}(\vec f))\right\|_{L^{q}\left(v_{\vec\omega}^{\frac{q}{p_0'}}\right)}\leq C\left\|M_{\delta}^{\#}(T_{\alpha}(\vec f))\right\|_{L^{q}\left(v_{\vec\omega}^{\frac{q}{p_0'}}\right)}\\ 
\leq &C\left\|\mathcal{M}_{\alpha,p_0'}(\vec f)\right\|_{L^{q}\left(v_{\vec\omega}^{\frac{q}{p_0'}}\right)}=C\left\|\left[\mathcal{M}_{\beta}(\vec f^{p_0'})\right]^{\frac{1}{p_0'}}\right\|_{L^{q}\left(v_{\vec\omega}^{\frac{q}{p_0'}}\right)}\\
=&C\left\|\mathcal{M}_{\beta}(\vec f^{p_0'})\right\|_{L^{\frac{q}{p_0'}}\left(v_{\vec\omega}^{\frac{q}{p_0'}}\right)}^{\frac{1}{p_0'}}\leq C\prod_{j=1}^{m}\left\||f_j|^{p_0'}\right\|_{L^{\frac{p_j}{p_0'}}\left(\omega_j^{\frac{p_j}{p_0'}}\right)}^{\frac{1}{p_0'}}\\
=&C\prod_{j=1}^m\|f_j\|_{L^{p_j}\left(\omega_j^{\frac{p_j}{p_0'}}\right)}.
\end{align*}
 Thus, Theorem \ref{Theorem 2.3} (i) is proved.

\item By Lemma \ref{Lemma2.1}, Theorem \ref{Theorem 2.1} and Lemma \ref{Lemma2.3} (ii), we obtain
\begin{align*}
\left\|T_{\alpha}(\vec f)\right\|_{L^{q,\infty}\left(v_{\vec\omega}^{\frac{q}{p_0'}}\right)}\leq &\left\|M_{\delta}(T_{\alpha}(\vec f))\right\|_{L^{q,\infty}\left(v_{\vec\omega}^{\frac{q}{p_0'}}\right)}\leq C\left\|M_{\delta}^{\#}(T_{\alpha}(\vec f))\right\|_{L^{q,\infty}\left(v_{\vec\omega}^{\frac{q}{p_0'}}\right)}\\
\leq &C\left\|{M}_{\alpha,p_0'}(\vec f)\right\|_{L^{q,\infty}\left(v_{\vec\omega}^{\frac{q}{p_0'}}\right)}=C\left\|\left[(\mathcal{M}_{\beta}({\vec f^{p_0'}})\right]^{\frac{1}{p_0'}}\right\|_{L^{q,\infty}\left(v_{\vec\omega}^{\frac{q}{p_0'}}\right)}\\
= &C\sup_{\lambda>0}\lambda\left|v_{\vec\omega}^{\frac{q}{p_0'}}\left(\left\{x\in{\mathbb{R}}^n:[(\mathcal{M}_{\beta}({\vec f^{p_0'}})]^{\frac{1}{p_0'}}>\lambda\right\}\right)\right|^{\frac{1}{q}}\\
= &C\sup_{\lambda>0}\lambda\left|v_{\vec\omega}^{\frac{q}{p_0'}}\left(\left\{x\in{\mathbb{R}}^n:\mathcal{M}_{\beta}({\vec f^{p_0'}})>\lambda^{p_0'}\right\}\right)\right|^{\frac{1}{q}}\\
= &C\sup_{\lambda>0}\lambda^{\frac{1}{p_0'}}\left|v_{\vec\omega}^{\frac{q}{p_0'}}\left(\left\{x\in{\mathbb{R}}^n:\mathcal{M}_{\beta}({\vec f^{p_0'}})>\lambda\right\}\right)\right|^{\frac{1}{q}}\\
= &C\left[\sup_{\lambda>0}\lambda\left|v_{\vec\omega}^{\frac{q}{p_0'}}\left(\left\{x\in{\mathbb{R}}^n:\mathcal{M}_{\beta}({\vec f^{p_0'}})>\lambda\right\}\right)\right|^{\frac{p_0'}{q}}\right]^{\frac{1}{p_0'}}\\
= &C\left\|\mathcal{M}_{\beta}({\vec f^{p_0'}})\right\|_{L^{\frac{q}{p_0'},\infty}\left(v_{\vec\omega}^{\frac{q}{p_0'}}\right)}^{\frac{1}{p_0'}}\leq C\prod_{j=1}^m\left\||f_j|^{p_0'}\right\|_{L^{\frac{p_j}{p_0'}}\left(\omega_j^{\frac{p_j}{p_0'}}\right)}^{\frac{1}{p_0'}}\\
= &C\prod_{j=1}^m\left\|f_j\right\|_{L^{p_j}\left(\omega_j^{\frac{p_j}{p_0'}}\right)}.
\end{align*}
\end{enumerate}
Thus, Theorem \ref{Theorem 2.3} (ii) is proved.
\end{proof}

\begin{theorem}\label{Theorem2.4} 
\rm Let $m\geq2$, $p_0'<t<\infty$, $0<\alpha t<mn$, $T_{\alpha}$ be an $m$-linear fractional integral operator with generalized kernel as defined in Definition \ref{definition1.2} and $\sum_{k=1}^{\infty}kC_k<\infty$. Suppose that $\vec b\in {\rm BMO}^m$, $\vec P=\left(p_1,\ldots,p_m\right)$, $\frac{1}{p}=\frac{1}{p_1}+\cdots+\frac{1}{p_m}$ with $t\leq p_j<\infty\left(j=1,\ldots,m\right)$, $\frac{1}{q}=\frac{1}{p}-\frac{\alpha}{n}>0$ and $\vec\omega\in A_{\frac{\vec P}{t},\frac{q}{t}}$.
\begin{itemize}
\item [\rm (i)]If $t<p_j<\infty$ for all $j=1,\ldots,m$, then 
\begin{align*}
\left\|T_{\vec b,\alpha}(\vec f)\right\|_{L^q\left(v_{\vec\omega}^{\frac{q}{t}}\right)}\leq C\big\|\vec b\big\|_{\left({\rm BMO}\right)^m}\prod_{j=1}^m\|f_j\|_{L^{p_j}\left(\omega_j^{\frac{p_j}{t}}\right)}.
\end{align*}

\item [\rm (ii)]If $t\leq p_j<\infty$ for all $j=1,\ldots,m$ and at least one of $p_j=t$, then 
\begin{align*}
\left\|T_{\vec b,\alpha}(\vec f)\right\|_{L^{q,\infty}\left(v_{\vec\omega}^{\frac{q}{t}}\right)}\leq C\big\|\vec b\big\|_{\left({\rm BMO}\right)^m}\prod_{j=1}^m\|f_j\|_{L^{p_j}\left(\omega_j^{\frac{p_j}{t}}\right)}.
\end{align*}
\end{itemize}
\end {theorem}

\begin{proof}
Firstly, we choose $\delta, \varepsilon$ such that $0<\delta<\varepsilon<\min\left\{1,\frac{sn}{n-s\alpha}\right\}$ and let $\beta=\alpha t$.
\begin{itemize}
\item [\rm (i)] By $\vec\omega\in A_{\frac{\vec P}{t},\frac{q}{t}}$, Lemma \ref{Lemma2.1}, Theorem \ref{Theorem2.2}, Theorem \ref{Lemma2.1} and Lemma \ref{Lemma2.3} (i), we conclude that

\begin{align*}
\left\|T_{\vec b,\alpha}(\vec f)\right\|_{L^q\left(v^{\frac{q}{t}}_{\vec\omega}\right)}\leq &\left\|M_{\delta}(T_{\vec b,\alpha}(\vec f))\right\|_{L^q\left(v^{\frac{q}{t}}_{\vec\omega}\right)}\leq \left\|M_{\delta}^{\#}(T_{\vec b,\alpha}(\vec f))\right\|_{L^q\left(v^{\frac{q}{t}}_{\vec\omega}\right)}\\
\leq  &C\big\|\vec b\big\|_{\left({\rm BMO}\right)^m}\left\|M_{\varepsilon}(T_{\alpha}(\vec f))+\mathcal{M}_{\alpha,t}(\vec f)\right\|_{L^q\left(v^{\frac{q}{t}}_{\vec\omega}\right)}\\
\leq &C\big\|\vec b\big\|_{\left({\rm BMO}\right)^m}\left(\left\|M_{\varepsilon}(T_{\alpha}(\vec f))\right\|_{L^q\left(v^{\frac{q}{t}}_{\vec\omega}\right)}+\left\|\mathcal{M}_{\alpha,t}(\vec f)\right\|_{L^q\left(v^{\frac{q}{t}}_{\vec\omega}\right)}\right)\\
\leq  &C\big\|\vec b\big\|_{\left({\rm BMO}\right)^m}\left(\left\|M_{\varepsilon}^{\#}(T_{\alpha}(\vec f))\right\|_{L^q\left(v^{\frac{q}{t}}_{\vec\omega}\right)}+\left\|\mathcal{M}_{\alpha,t}(\vec f)\right\|_{L^q\left(v^{\frac{q}{t}}_{\vec\omega}\right)}\right)\\
\leq  &C\big\|\vec b\big\|_{\left({\rm BMO}\right)^m}\left(\left\|\mathcal{M}_{\alpha,p_0'}(\vec f)\right\|_{L^q\left(v^{\frac{q}{t}}_{\vec\omega}\right)}+\left\|\mathcal{M}_{\alpha,t}(\vec f)\right\|_{L^q\left(v^{\frac{q}{t}}_{\vec\omega}\right)}\right)\\
\leq &C\big\|\vec b\big\|_{\left({\rm BMO}\right)^m}\left\|\mathcal{M}_{\alpha,t}(\vec f)\right\|_{L^q\left(v^{\frac{q}{t}}_{\vec\omega}\right)}= C\big\|\vec b\big\|_{\left({\rm BMO}\right)^m}\left\|\left[\mathcal{M}_{\beta}(\vec f^{t})\right]^{\frac{1}{t}}\right\|_{L^q\left(v^{\frac{q}{t}}_{\vec\omega}\right)}\\
=  &C\big\|\vec b\big\|_{\left({\rm BMO}\right)^m}\left\|\mathcal{M}_{\beta}(\vec f^{t})\right\|_{L^{\frac{q}{t}}\left(v^{\frac{q}{t}}_{\vec\omega}\right)}^{\frac{1}{t}}\leq C\big\|\vec b\big\|_{\left({\rm BMO}\right)^m}\prod_{j=1}^m\left\||f_j|^t\right\|_{L^{\frac{p_j}{t}}\left(\omega_j^{\frac{p_j}{t}}\right)}^{\frac{1}{t}}\\
=  &C\big\|\vec b\big\|_{\left({\rm BMO}\right)^m}\prod_{j=1}^m\|f_j\|_{L^{p_j}\left(\omega_j^{\frac{p_j}{t}}\right)}.
\end{align*}
\noindent Thus, Theorem \ref{Theorem2.4} (i) is proved.

\item [\rm (ii)] By Lemma \ref{Lemma2.1}, Theorem \ref{Theorem2.2}, Theorem \ref{Theorem 2.1} and Lemma \ref{Lemma2.3} (ii), we conclude that
\begin{align*}
\left\|T_{\vec b,\alpha}(\vec f)\right\|_{L^{q,\infty}\left(v_{\vec\omega}^{\frac{q}{t}}\right)}&\leq \left\|M_{\delta}(T_{\vec b,\alpha}(\vec f))\right\|_{L^{q,\infty}\left(v_{\vec\omega}^{\frac{q}{t}}\right)}\leq C\left\|M_{\delta}^{\#}(T_{\vec b,\alpha}(\vec f))\right\|_{L^{q,\infty}\left(v_{\vec\omega}^{\frac{q}{t}}\right)}\\
&\leq C\big\|\vec b\big\|_{{({\rm BMO})}^m}\left\|M_{\varepsilon}(T_{\alpha}(\vec f))+{M}_{\alpha,t}(\vec f)\right\|_{L^{q,\infty}\left(v_{\vec\omega}^{\frac{q}{t}}\right)}\\
&\leq C\big\|\vec b\big\|_{{({\rm BMO})}^m}\left(\left\|M_{\varepsilon}(T_{\alpha}(\vec f))\right\|_{L^{q,\infty}\left(v_{\vec\omega}^{\frac{q}{t}}\right)}+\left\|{M}_{\alpha,t}(\vec f)\right\|_{L^{q,\infty}\left(v_{\vec\omega}^{\frac{q}{t}}\right)}\right)\\
&\leq C\big\|\vec b\big\|_{({\rm BMO})^m}\left(\left\|M^{\#}_{\varepsilon}(T_{\alpha}(\vec f))\right\|_{L^{q,\infty}\left(v_{\vec\omega}^{\frac{q}{t}}\right)}+\left\|{M}_{\alpha,t}(\vec f)\right\|_{L^{q,\infty}\left(v_{\vec\omega}^{\frac{q}{t}}\right)}\right)\\
&\leq C\big\|\vec b\big\|_{({\rm BMO})^m}\left(\left\|{M}_{\alpha,p_0'}(\vec f)\right\|_{L^{q,\infty}\left(v_{\vec\omega}^{\frac{q}{t}}\right)}+\left\|{M}_{\alpha,t}(\vec f)\right\|_{L^{q,\infty}\left(v_{\vec\omega}^{\frac{q}{t}}\right)}\right)\\
&\leq C\big\|\vec b\big\|_{({\rm BMO})^m}\left\|{M}_{\alpha,t}(\vec f)\right\|_{L^{q,\infty}\left(v_{\vec\omega}^{\frac{q}{t}}\right)}\\
&=C\big\|\vec b\big\|_{({\rm BMO})^m}\left\|\left[\mathcal{M}_{\beta}({\vec f^{t}})\right]^{\frac{1}{t}}\right\|_{L^{q,\infty}\left(v_{\vec\omega}^{\frac{q}{t}}\right)}\\
&= C\big\|\vec b\big\|_{({\rm BMO})^m}\sup_{\lambda>0}\lambda\left|v_{\vec\omega}^{\frac{q}{t}}\left(\left\{x\in{\mathbb{R}}^n:\left[\mathcal{M}_{\beta}({\vec f^{t}})\right]^{\frac{1}{t}}(x)>\lambda\right\}\right)\right|^{\frac{1}{q}}\\
&= C\big\|\vec b\big\|_{({\rm BMO})^m}\sup_{\lambda>0}\lambda\left|v_{\vec\omega}^{\frac{q}{t}}\left(\left\{x\in{\mathbb{R}}^n:\mathcal{M}_{\beta}({\vec f^{t}})>\lambda^{t}\right\}\right)\right|^{\frac{1}{q}}\\
&= C\big\|\vec b\big\|_{({\rm BMO})^m}\sup_{\lambda>0}\lambda^{\frac{1}{t}}\left|v_{\vec\omega}^{\frac{q}{t}}\left(\left\{x\in{\mathbb{R}}^n:\mathcal{M}_{\beta}({\vec f^{t}})>\lambda\right\}\right)\right|^{\frac{1}{q}}\\
&= C\big\|\vec b\big\|_{({\rm BMO})^m}\left[\sup_{\lambda>0}\lambda\left|v_{\vec\omega}^{\frac{q}{t}}\left(\left\{x\in{\mathbb{R}}^n:\mathcal{M}_{\beta}({\vec f^{t}})>\lambda\right\}\right)\right|^{\frac{t}{q}}\right]^{\frac{1}{t}}\\
&=C\big\|\vec b\big\|_{({\rm BMO})^m}\left\|\mathcal{M}_{\beta}({\vec f^{t}})\right\|_{L^{\frac{q}{t},\infty}\left(v_{\vec\omega}^{\frac{q}{t}}\right)}^{\frac{1}{t}}\\
&\leq C\big\|\vec b\big\|_{({\rm BMO})^m}\prod_{j=1}^m\left\||f_j|^{t}\right\|_{L^{\frac{p_j}{t}}\left(\omega_j^{\frac{p_j}{t}}\right)}^{\frac{1}{t}}\\
&=C\big\|\vec b\big\|_{({\rm BMO})^m}\prod_{j=1}^m\|f_j\|_{L^{p_j}\left(\omega_j^{\frac{p_j}{t}}\right)}.
\end{align*}
\end{itemize}
Thus, Theorem \ref{Theorem2.4} (ii) is proved.
\end{proof}

\begin{remark}\label{remark 2.1}
\rm When the generalized kernel condition goes back to the Dini kernel condition, the boundedness of $m$-linear fractional integral operators with Dini's type kernels on Lebesgue spaces has been demonstrated by Wu and Zhang in \cite{WZ2023}.
\end{remark}

\begin{corollary}\label{Corollary 2.1}
\rm When the generalized kernel condition goes back to the Dini kernel condition, the boundedness of multilinear commutators of $m$-linear fractional integral operators with Dini's type kernels on Lebesgue spaces remains new.
\end{corollary}

\section{The Boundedness on Variable Exponent Lebesgue Spaces}\label{sec3}

\quad\quad Firstly, we introduce several definitions and notations that will be utilized subsequently. Then we establish the boundedness of multilinear fractional integral operators with generalized kernels and multilinear commutators on variable exponent Lebesgue spaces, respectively.

\begin{definition}\label{definition3.1}{\rm(see \cite{CF2006})}
\rm Let $q{\left(\cdot\right)}:{\mathbb{R}}^n\rightarrow[1,\infty)$ be a measurable function.
The variable exponent Lebesgue space $L^{q{\left(\cdot\right)}}\left({\mathbb{R}}^n\right)$ is defined by
\begin{align*}
L^{q{\left(\cdot\right)}}\left({\mathbb{R}}^n\right):=\{f~{\rm is~measurable~function}: F_q\left(f/\eta\right)<\infty~{\rm for~some~constant}~\eta>0\},
\end{align*}
where $F_q\left(f\right):=\int_{{\mathbb{R}}^n}|f\left(x\right)|^{q\left(x\right)}dx$ is a convex functional modular.

The space $L^{q{\left(\cdot\right)}}\left({\mathbb{R}}^n\right)$ is a Banach function space with respect to the Luxemburg type norm.
\begin{align*}
\|f\|_{L^{q{\left(\cdot\right)}}\left({\mathbb{R}}^n\right)}:=\inf\{\eta>0:\int_{{\mathbb{R}}^n}\left({\frac{|f\left(x\right)|}{\eta}}\right)^{q\left(x\right)}dx\leq 1\}.
\end{align*}
\end{definition}

\begin{definition}\label{definition3.2}{\rm(see \cite{CF2006})}
\rm Given a measurable function $q\left({\cdot}\right)$ defined on ${\mathbb{R}}^n$. For $E\subset {\mathbb{R}}^n$, we write
\begin{align*}
q_{-}\left(E\right):=\mathop{\rm ess~inf}\limits_{x\in E}q\left(x\right),~~q_{+}\left(E\right):=\mathop{\rm ess~inf}\limits_{x\in E}q\left(x\right),
\end{align*}
and write $q_{-}\left({\mathbb{R}}^n\right)=q_{-}$ and $q_{+}\left({\mathbb{R}}^n\right)=q_{+}$ simply.
\begin{itemize}
\item [\rm (i)]$q'_{-}=\mathop{\rm ess~inf}\limits_{x\in {\mathbb{R}}^n}q'\left(x\right)=\frac{q_+}{q_+-1}$, $q'_{+}=\mathop{\rm ess~inf}\limits_{x\in {\mathbb{R}}^n}q'\left(x\right)=\frac{q_-}{q_-+1}$.
\item [\rm (ii)]Denote by $\mathscr{P}\left({\mathbb{R}}^n\right)$ the set of all measurable functions $q\left(\cdot\right):{\mathbb{R}}^n\rightarrow \left(1,\infty\right)$ such that
\begin{align*}
1<q_{-}\leq q\left(x\right)\leq q_+<\infty, x\in {\mathbb{R}}^n.
\end{align*}
\item [\rm (iii)]The set $\mathscr{B}\left({\mathbb{R}}^n\right)$ consists of all measurable functions $q\left(\cdot\right)\in\mathscr{P}\left({\mathbb{R}}^n\right)$ such that the Hardy-Littlewood maximal operator $M$ is bounded on $L^{q\left(\cdot\right)}\left({\mathbb{R}}^n\right)$.
\begin{align*}
\|Mf\|_{L^{q\left(\cdot\right)}}\leq \|f\|_{L^{q\left(\cdot\right)}}.
\end{align*}
\end{itemize}
\end{definition}

\begin{definition}\label{definition3.3}{\rm(see \cite{CF2013})}
\rm  Let $q\left(\cdot\right)$ be a real-valued function on ${\mathbb{R}}^n$.
\begin{itemize}
\item [\rm (i)]Denote by $\mathscr{L}_{loc}^{log}\left({\mathbb{R}}^n\right)$ the set of all local $\log$-H\"{o}lder continuous functions $q\left(\cdot\right)$ which satisfy
\begin{align*}
|q\left(x\right)-q\left(y\right)|\leq \frac{-C}{\ln\left(|x-y|\right)},|x-y|\leq 1/2, x,y\in {\mathbb{R}}^n.
\end{align*}
Here, the positive constant $C$ does not depend on $x$ or $y$.
\item [\rm (ii)]The set $\mathscr{L}_{\infty}^{log}\left({\mathbb{R}}^n\right)$ consists of all $\log$-H\"{o}lder continuous functions $q\left(\cdot\right)$ at  infinity, satisfying 
\begin{align*}
|q\left(x\right)-q_{\infty}|\leq \frac{C_{\infty}}{\ln\left(e+|x|\right)}, x\in {\mathbb{R}}^n,
\end{align*}
where $q_{\infty}:=\lim_{|x|\rightarrow\infty}q\left(x\right)$.
\item [\rm (iii)]Denote by $\mathscr{L}^{log}\left({\mathbb{R}}^n\right):=\mathscr{L}_{loc}^{log}\left({\mathbb{R}}^n\right)\cap \mathscr{L}_{\infty}^{log}\left({\mathbb{R}}^n\right)$ the set of all global $\log$-H\"{o}lder continuous functions $q\left(\cdot\right)$.
\end{itemize}
\end{definition}

\begin{remark}\label{remark 3.1}{\rm(see \cite{WZ2023})}
\rm The condition $\mathscr{L}_{\infty}^{log}\left({\mathbb{R}}^n\right)$ is equivalent to the uniform continuity condition
\begin{align*}
|q\left(x\right)-q\left(y\right)|\leq \frac{C}{\ln\left(e+|x|\right)},|y|\geq |x|, x,y\in {\mathbb{R}}^n.
\end{align*}

In what follows, we denote
\begin{align*}
\mathscr{P}^{log}\left({\mathbb{R}}^n\right):=\mathscr{L}^{log}\left({\mathbb{R}}^n\right)\cap \mathscr{P}\left({\mathbb{R}}^n\right).
\end{align*}
\end{remark}

\begin{lemma}\label{Lemma3.1}{\rm(see \cite{CF2006})} 
\rm Let $p\left(\cdot\right)\in \mathscr{P}\left({\mathbb{R}}^n\right)$.
\begin{itemize}
\item [\rm (1)]If $p\left(\cdot\right)\in \mathscr{L}^{log}\left({\mathbb{R}}^n\right)$, then we have $p\left(\cdot\right)\in \mathscr{B}\left({\mathbb{R}}^n\right)$.
\item [\rm (2)]The following condition are equivalent:
\begin{itemize}
\item [\rm (i)]$p\left(\cdot\right)\in \mathscr{B}\left({\mathbb{R}}^n\right)$.
\item [\rm (ii)]$\left(p\left(\cdot\right)/p_0\right)'\in \mathscr{B}\left({\mathbb{R}}^n\right)$ for some $1<p_0<p_{-}$.
\end{itemize}
\end{itemize}
\end{lemma}

\begin{lemma}\label{Lemma3.2}{\rm(see \cite{LZ2014})} 
\rm Let $q\left(\cdot\right)$, $q_{1}\left(\cdot\right),\ldots,q_{m}\left(\cdot\right)\in\mathscr{P}\left({\mathbb{R}}^n\right)$ satisfy the condition
\begin{align*}
\frac{1}{q\left(x\right)}=\frac{1}{q_1\left(x\right)}+\cdots+\frac{1}{q_m\left(x\right)},~for~a.e~x\in {\mathbb{R}}^n.
\end{align*}
Then, for any $f_j\in L^{q_j\left(\cdot\right)}\left({\mathbb{R}}^n\right)$, $j=1,\ldots,m$, one has
\begin{align*}
\|f_1\cdots f_m\|_{q\left(\cdot\right)}\leq C\|f_1\|_{q_1\left(\cdot\right)}\cdots\|f_m\|_{q_m\left(\cdot\right)}.
\end{align*}
\end{lemma}

\begin{lemma}\label{Lemma3.3}{\rm(see \cite{DH2011})} 
\rm Let $\mathcal{F}$ denote a family of ordered pairs of measurable functions $\left(f,g\right)$. suppose that for some fixed $q_0$ with $0<q_0<\infty$ and every weight $\omega\in A_1$ such that
\begin{align*}
\int_{{\mathbb{R}}^n}|f\left(x\right)|^{q_0}\omega\left(x\right)dx\leq C_0\int_{{\mathbb{R}}^n}|g\left(x\right)|^{q_0}\omega\left(x\right)dx.
\end{align*}
Let $q\left(\cdot\right)\in \mathscr{P}\left({\mathbb{R}}^n\right)$ with $q_0\leq q_{-}$. If $\left(q\left(\cdot\right)/q_0\right)'\in \mathscr{B}\left({\mathbb{R}}^n\right)$, then there exists a positive constant $C$, such that for all $\left(f,g\right)\in \mathcal{F}$,
\begin{align*}
\|f\|_{q\left(\cdot\right)}\leq C\|g\|_{q\left(\cdot\right)}.
\end{align*}
\end{lemma}

\begin{lemma}\label{Lemma3.4}{\rm(see \cite{CCF2007})} 
\rm Let $0<\alpha<n$ and $p\left(\cdot\right)\in\mathscr{P}^{log}\left({\mathbb{R}}^n\right)$ with $p_{+}<\frac{n}{\alpha} $ and $\frac{1}{q\left(\cdot\right)}=\frac{1}{p\left(\cdot\right)}-\frac{\alpha}{n}$, then
\begin{align*}
\|\mathscr{M}_{\alpha}f\|_{q\left(\cdot\right)}\leq C\|f\|_{p\left(\cdot\right)}.
\end{align*}
\end{lemma}

\begin{theorem}\label{Theorem 3.1}
\rm Let $m\geq2$, $T_{\alpha}$ be an $m$-linear fractional integral operator with generalized kernel as Definition \ref{definition1.2} and $\sum_{k=1}^{\infty}C_k<\infty$. Given $\frac{p_i}{p_0'}\left(\cdot\right)\in \mathscr{P}^{log}\left({\mathbb{R}}^n\right)$, $0<\alpha_i<\frac{n}{p_{i,+}}$, $i=1,2,\ldots,m$, $\alpha=\alpha_1+\cdots+\alpha_m$, $\frac{1}{p\left(\cdot\right)}=\frac{1}{p_1\left(\cdot\right)}+\cdots+\frac{1}{p_m\left(\cdot\right)}$, $\frac{1}{q\left(\cdot\right)}=\frac{1}{p\left(\cdot\right)}-\frac{\alpha}{n}$ and $q\left(\cdot\right)\in\mathscr{B}\left({\mathbb{R}}^n\right)$, then there exists a positive constant $C$ such that 
\begin{align*}
\left\|T_{\alpha}(\vec f)\right\|_{L^{q\left(\cdot\right)}\left({\mathbb{R}}^n\right)}\leq C\prod_{j=1}^m\left\|f_j\right\|_{L^{p_j\left(\cdot\right)}\left({\mathbb{R}}^n\right)}.
\end{align*}
\end{theorem}

\begin{proof}
Since $q\left(\cdot\right)\in\mathscr{B}\left({\mathbb{R}}^n\right)$, by Lemma \ref{Lemma3.1}, there exists a $q_0$ with $1<q_0<q_-$ such that $\left(q\left(\cdot\right)/q_0\right)'\in \mathscr{B}\left({\mathbb{R}}^n\right)$. For this $q_0$ and any $\omega\in A_1$, choosing $0<\delta<\min\left\{1,\frac{sn}{n-s\alpha}\right\}$, by Theorem \ref{Theorem 2.1}, we obtain
\begin{align*}
\left\|T_{\alpha}(\vec f)\right\|_{L^{q_0}\left(\omega\right)}\leq \left\|M_{\delta}(T_{\alpha}(\vec f))\right\|_{L^{q_0}\left(\omega\right)}\leq \left\|M_{\delta}^{\#}(T_{\alpha}(\vec f))\right\|_{L^{q_0}\left(\omega\right)}\leq C\left\|\mathcal{M}_{\alpha,p_0'}(\vec f)\right\|_{L^{q_0}\left(\omega\right)},
\end{align*}
then we can get
\begin{align*}
\int_{{\mathbb{R}}^n}\left|T_{\alpha}(\vec f)\left(x\right)\right|^{q_0}\omega\left(x\right)dx\leq C\int_{{\mathbb{R}}^n}\left|\mathcal{M}_{\alpha,p_0'}(\vec f)\left(x\right)\right|^{q_0}\omega\left(x\right)dx
\end{align*}
holds for all $m$-tuples $\vec f=\left(f_1,\ldots,f_m\right)$ of bounded functions with compact support.

  Apply Lemma \ref{Lemma3.3} to the pair $\left(T_{\alpha}(\vec f),M_{\alpha,p_0'}(\vec f)\right)$ and obtain
\begin{equation}\label{(3.1)}
\left\|T_{\alpha}(\vec f)\right\|_{q\left(\cdot\right)}\leq C\left\|\mathcal{M}_{\alpha,p_0'}(\vec f)\right\|_{q\left(\cdot\right)}.
\end{equation}
We can get $0<\alpha_i<\frac{n}{p_{i,+}}<\frac{n}{p_0'}$. Denote by $\frac{1}{q_i\left(\cdot\right)}=\frac{1}{p_i\left(\cdot\right)}-\frac{\alpha_i}{n}\left(i=1,\ldots,m\right)$, then $\frac{1}{q\left(\cdot\right)}=\frac{1}{q_1\left(\cdot\right)}+\cdots+\frac{1}{q_m\left(\cdot\right)}$. By Definition \ref{definition2.3}, it is easy to see that
\begin{equation}\label{(3.2)}
\mathcal{M}_{\alpha,p_0'}(\vec f)\left(x\right)\leq\prod_{i=1}^m\mathscr{M}_{\alpha_i,p_0'}\left(f_i\right)\left(x\right)~{\rm for}~x\in{\mathbb{R}}^n.
\end{equation}
Let $\beta_i=\alpha_i p_0'$. From \eqref{(3.1)}, \eqref{(3.2)}, Lemma \ref{Lemma3.2}, and Lemma \ref{Lemma3.4}, it follows that
\begin{align*}
\left\|T_{\alpha}(\vec f)(x)\right\|_{q(\cdot)}\leq &C\left\|\mathcal{M}_{\alpha,p_0'}(\vec f)\right\|_{q(\cdot)}\leq C\prod_{i=1}^m\left\|\mathscr{M}_{\alpha_i,p_0'}(f_i)\right\|_{q_i(\cdot)}\\
=&C\prod_{i=1}^m\left\|\left[\mathscr{M}_{{\beta}_i}(f_i^{p_0'})\right]^{\frac{1}{p_0'}}\right\|_{q_i(\cdot)}=C\prod_{i=1}^m\left\|\mathscr{M}_{{\beta}_i}(f_i^{p_0'})\right\|_{\frac{q_i}{p_0'}(\cdot)}^{\frac{1}{p_0'}}\\
\leq &C\prod_{i=1}^m\left\||f_i|^{p_0'}\right\|_{\frac{p_i}{p_0'}(\cdot)}^{\frac{1}{p_0'}}=C\prod_{i=1}^m\left\|f_i\right\|_{p_i(\cdot)}.
\end{align*}
Thus, Theorem \ref{Theorem 3.1} is proved.
\end{proof}

\begin{remark}\label{remark 3.2}
\rm  When the generalized kernel condition goes back to the Dini kernel condition, the boundedness of $m$-linear fractional integral operators with Dini's type kernels on variable exponent Lebesgue spaces has been demonstrated by Wu and Zhang in \cite{WZ2023}.
\end{remark}

\begin{theorem}\label{Theorem 3.2}
\rm Let $m\geq2$, $T_{\alpha}$ be an $m$-linear fractional integral operator with generalized kernel as Definition \ref{definition1.2} and $\sum_{k=1}^{\infty}kC_k<\infty$. Given $\frac{p_i}{p_0'}\left(\cdot\right)\in \mathscr{P}^{log}\left({\mathbb{R}}^n\right)$, $0<\alpha_i<\frac{n}{p_{i,+}}$, $i=1,2,\ldots,m$, $\alpha=\alpha_1+\cdots+\alpha_m$, $\frac{1}{p\left(\cdot\right)}=\frac{1}{p_1\left(\cdot\right)}+\cdots+\frac{1}{p_m\left(\cdot\right)}$, $\frac{1}{q\left(\cdot\right)}=\frac{1}{p\left(\cdot\right)}-\frac{\alpha}{n}$, $q\left(\cdot\right)\in\mathscr{B}\left({\mathbb{R}}^n\right)$ and $\vec b\in {\rm BMO}^m$, then there exists a positive constant $C$ such that 
\begin{align*}
\left\|T_{\vec b,\alpha}(\vec f)\right\|_{L^{q\left(\cdot\right)}\left({\mathbb{R}}^n\right)}\leq C\big\|\vec b\big\|_{\left({\rm BMO}\right)^m}\prod_{j=1}^m\|f_j\|_{L^{p_j\left(\cdot\right)}\left({\mathbb{R}}^n\right)}.
\end{align*}
\end{theorem}

\begin{proof}
Since $q\left(\cdot\right)\in\mathscr{B}\left({\mathbb{R}}^n\right)$, by Lemma \ref{Lemma3.1} there exists a $q_0$ with $1<q_0<q_-$ such that $\left(q\left(\cdot\right)/q_0\right)'\in \mathscr{B}\left({\mathbb{R}}^n\right)$.   For this $q_0$ and any $\omega\in A_1$, choosing $0<\delta<\varepsilon<\min\left\{1,\frac{sn}{n-s\alpha}\right\}$ and $p_0'<t<\min\limits_{1\leq i\leq m}p_{i,-}$, by Theorem \ref{Theorem2.2}, we obtain
\begin{align*}
\left\|T_{\vec b,\alpha}(\vec f)\right\|_{L^{q_0}\left(\omega\right)}\leq &\left\|M_{\delta}(T_{\vec b,\alpha}(\vec f))\right\|_{L^{q_0}\left(\omega\right)}\leq \left\|M_{\delta}^{\#}(T_{\vec b,\alpha}(\vec f))\right\|_{L^{q_0}\left(\omega\right)}\\
\leq &C\big\|\vec b\big\|_{\left({\rm BMO}\right)^m}\left\|M_{\varepsilon}(T_{\alpha}(\vec f))+\mathcal{M}_{\alpha,t}(\vec f)\right\|_{L^{q_0}\left(\omega\right)}\\
\leq &C\big\|\vec b\big\|_{\left({\rm BMO}\right)^m}\left(\left\|M_{\varepsilon}(T_{\alpha}(\vec f))\right\|_{L^{q_0}\left(\omega\right)}+\left\|\mathcal{M}_{\alpha,t}(\vec f)\right\|_{L^{q_0}\left(\omega\right)}\right)\\
\leq &C\big\|\vec b\big\|_{\left({\rm BMO}\right)^m}\left(\left\|M_{\varepsilon}^{\#}(T_{\alpha}(\vec f))\right\|_{L^{q_0}\left(\omega\right)}+\left\|\mathcal{M}_{\alpha,t}(\vec f)\right\|_{L^{q_0}\left(\omega\right)}\right)\\
\leq &C\big\|\vec b\big\|_{\left({\rm BMO}\right)^m}\left(\left\|\mathcal{M}_{\alpha,p_0'}(\vec f)\right\|_{L^{q_0}\left(\omega\right)}+\left\|\mathcal{M}_{\alpha,t}(\vec f)\right\|_{L^{q_0}\left(\omega\right)}\right)\\
\leq &C\big\|\vec b\big\|_{\left({\rm BMO}\right)^m}\left\|\mathcal{M}_{\alpha,t}(\vec f)\right\|_{L^{q_0}\left(\omega\right)},
\end{align*}
then we can get
\begin{align*}
\int_{{\mathbb{R}}^n}\left|T_{\vec b,\alpha}(\vec f)\left(x\right)\right|^{q_0}\omega\left(x\right)dx\leq C\big\|\vec b\big\|_{\left({\rm BMO}\right)^m}^{q_0}\int_{{\mathbb{R}}^n}\left|\mathcal{M}_{\alpha,t}(\vec f)\left(x\right)\right|^{q_0}\omega\left(x\right)dx
\end{align*}
holds for all $m$-tuples $\vec f=\left(f_1,\ldots,f_m\right)$ of bounded functions with compact support.

  Apply Lemma \ref{Lemma3.3} to the pair $\left(T_{\vec b,\alpha}(\vec f),M_{\alpha,t}(\vec f)\right)$ and obtain
\begin{equation}\label{(3.3)}
\left\|T_{\vec b,\alpha}(\vec f)\right\|_{q\left(\cdot\right)}\leq C\big\|\vec b\big\|_{\left({\rm BMO}\right)^m}\left\|\mathcal{M}_{\alpha,t}(\vec f)\right\|_{q\left(\cdot\right)}.
\end{equation}
We can get $0<\alpha_i<\frac{n}{p_{i,+}}<\frac{n}{t}$. Denote by $\frac{1}{q_i\left(\cdot\right)}=\frac{1}{p_i\left(\cdot\right)}-\frac{\alpha_i}{n}\left(i=1,\ldots,m\right)$, then $\frac{1}{q\left(\cdot\right)}=\frac{1}{q_1\left(\cdot\right)}+\cdots+\frac{1}{q_m\left(\cdot\right)}$. By Definition \ref{definition2.3}, it is easy to see that
\begin{equation}\label{(3.4)}
\mathcal{M}_{\alpha,t}(\vec f)\left(x\right)\leq\prod_{i=1}^m\mathscr{M}_{\alpha_i,t}\left(f_i\right)\left(x\right)~{\rm for} ~x\in{\mathbb{R}}^n.
\end{equation}
Let $\beta_i=\alpha_i t$. The fact $\frac{p_i}{p_0'}\left(\cdot\right)\in \mathscr{P}^{log}\left({\mathbb{R}}^n\right)$ implies that $\frac{p_i}{t}\left(\cdot\right)\in \mathscr{P}^{log}\left({\mathbb{R}}^n\right)$, $i=1,\ldots,m$. From \eqref{(3.3)}, \eqref{(3.4)}, Lemma \ref{Lemma3.2} and Lemma \ref{Lemma3.4}, it follows that
\begin{align*}
\left\|T_{\vec b,\alpha}(\vec f)(x)\right\|_{q(\cdot)}\leq &C\big\|\vec b\big\|_{({\rm BMO})^m}\left\|\mathcal{M}_{\alpha,t}(\vec f)\right\|_{q(\cdot)}\leq C\big\|\vec b\big\|_{({\rm BMO})^m}\prod_{i=1}^m\left\|\mathscr{M}_{\alpha_i,t}(f_i)\right\|_{q_i(\cdot)}\\
=&C\big\|\vec b\big\|_{({\rm BMO})^m}\prod_{i=1}^m\left\|\left[\mathscr{M}_{{\beta}_i}(f_i^{t})\right]^{\frac{1}{t}}\right\|_{q_i(\cdot)}=C\big\|\vec b\big\|_{({\rm BMO})^m}\prod_{i=1}^m\left\|\mathscr{M}_{{\beta}_i}(f_i^{t})\right\|_{\frac{q_i}{t}(\cdot)}^{\frac{1}{t}}\\
\leq &C\big\|\vec b\big\|_{({\rm BMO})^m}\prod_{i=1}^m\left\||f_i|^{t}\right\|_{\frac{p_i}{t}(\cdot)}^{\frac{1}{t}}=C\big\|\vec b\big\|_{({\rm BMO})^m}\prod_{i=1}^m\left\|f_i\right\|_{p_i(\cdot)}.
\end{align*}
Thus, Theorem \ref{Theorem 3.2} is proved.
\end{proof}

\begin{corollary}\label{Corollary 3.1}
\rm When the generalized kernel condition goes back to the Dini kernel condition, the boundedness of multilinear commutators generalized by $m$-linear fractional integral operators with Dini's type kernels and ${\rm BMO}$ functions on variable exponent Lebesgue spaces remains new.
\end{corollary}

\bigskip

\noindent  

\smallskip

\noindent School of Science, China University of Mining and
Technology, Beijing 100083,  People's Republic of China

\smallskip

\noindent{\it E-mails:} \texttt{linyan@cumtb.edu.cn} (Y. Lin)

\noindent\phantom{{\it E-mails:} }\texttt{zhaoyuhang@student.cumtb.edu.cn} (Y. Zhao)

\noindent\phantom{{\it E-mails:} }\texttt{yangshuhui@student.cumtb.edu.cn} (S. Yang)

\end{document}